\documentclass[12pt,reqno]{amsart}
\usepackage{amssymb, amsmath,amsthm}
\usepackage{graphicx}
\newtheorem{thm}{Theorem}

\newtheorem{defn}{Definition}

\numberwithin{defn}{section}
\numberwithin{thm}{section}
\numberwithin{Lemma}{section}
\numberwithin{Corollary}{section}
\numberwithin{Example}{section}
\numberwithin{subsection}{section}
\numberwithin{Remark}{section}
\numberwithin{equation}{section}
\numberwithin{ppn}{section}
\begin{document}
\title[ A class of iterative methods ... ]
{A class of iterative methods for solving nonlinear equation with fourth-order convergence} 
\author{J. P. Jaiswal }
\date{}
\maketitle


\textbf{Abstract.} 
In this paper we established a class of optimal fourth-order methods  which is obtained by existing third-order method for solving nonlinear equations for simple roots by using weight functions. Some physical examples are given to illustrate the efficiency and performance of our method.\\

\\


\section{Introduction}
Nonlinear equations plays an important role in science and engineering. Finding an analytic solution is not always possible. Therefore, numerical methods are used in such situations. The classical Newton's method is the best known iterative method for solving nonlinear equations. To improve the local order of convergence and efficiency index, many modified third-order methods have been presented in literature. For detail we refer [\cite{Weerakoon}, \cite{Homeier2}, \cite{Ozban}, \cite{Wang}, \cite{Jain}] and references therein.
 
Recently Ardelean \cite{Ardelean} established \textit{Bisectrix Newton's Method (BN)},  which is given by
 \begin{eqnarray}\label{eqn:11a}
 x_{n+1}=x_n-\frac{(f'(x_n)+f'(y_n))f(x_n)}{f'(x_n)f'(y_n)+\sqrt{(1+f'(x_n)^2)(1+f'(y_n)^2)}-1}, n\geq 0,
 \end{eqnarray}
where $y_n=x_n-\frac{f(x_n)}{f'(x_n)}$. The method is third-order convergent for simple roots and its efficiency index is $3^{1/3}=1.442$.

Weerakoon et al. \cite{Weerakoon} used Newton's theorem
\begin{eqnarray}\label{eqn:11b}
f(x)=f(x_n)+\int_{x_n}^{x}f'(t)dt
\end{eqnarray}
and approximated the integral by trapezoidal rule, i.e.
\begin{eqnarray}\label{eqn:11c}
\int_{x_n}^{x}f'(t)dt=\frac{(x-x_n)}{2}[f'(x_n)+f'(x)],
\end{eqnarray}
obtained the following variant of the Newton method
\begin{eqnarray}\label{eqn:11d}
 x_{n+1}=x_n-\frac{2f(x_n)}{f'(x_n)+f'(y_n)},
\end{eqnarray}
where $y_n=x_n-\frac{f(x_n)}{f'(x_n)}$. It is shown that this is third-order. Many authors used this idea by approximating the integral $\int_{x_n}^{x}f'(t)dt$ by different ways. For more detail, one can see [ \cite{Frontini1}, \cite{Frontini2}, \cite{Frontini3},  \cite{Homeier1}, \cite{Homeier2}, \cite{Wang}, \cite{Jain} ] and the references there in.
If we approximated the integral  in $(\ref{eqn:11b})$ by
\begin{eqnarray}\label{eqn:11e}
\int_{x_n}^{x}f'(t)dt=(x-x_n)\left[\frac{f'(x_n)f'(x)+\sqrt{(1+f'(x_n)^2)(1+f'(x)^2)}-1}{(f'(x_n)+f'(x))}\right],
 \end{eqnarray}
we get the same formula $(\ref{eqn:11a})$. 

Next Homeier \cite{Homeier2} used Newton's theorem $(\ref{eqn:11b})$ for the inverse function $x=f^{-1}(y)=g(y)$ instead of $y=f(x)$, that is 
\begin{eqnarray}\label{eqn:11f}
g(y)&=& g(y_n)+\int_{y_n}^{y}g'(s)ds. 
\end{eqnarray}
Then the method $(\ref{eqn:11d})$ takes the form 
\begin{eqnarray}\label{eqn:11g}
 x_{n+1}=x_n-\frac{f(x_n)}{2}\left[\frac{1}{f'(x_n)}+ \frac{1}{f'(y_n)}\right],
\end{eqnarray}
where $y_n=x_n-\frac{f(x_n)}{f'(x_n)}$. This method is again third-order.

Here we state following  definitions:
\begin{defn}
Let f(x) be a real function with a simple root $\alpha$ and let ${x_n}$ be a sequence of real numbers that converge towards $\alpha$. The order of convergence m is given by
\begin{equation}\label{eqn:21}
\lim_{n\rightarrow\infty}\frac{x_{n+1}-\alpha}{(x_n-\alpha)^m}=\zeta\neq0,  
\end{equation}     
\noindent
where $\zeta$ is the asymptotic error constant and $m \in R^+$.\\
\end{defn}

\begin{defn}
Let $\beta$ be the number of function evaluations of the new method. The efficiency of the new method is measured by the concept  of efficiency index \cite{Gautschi,Traub1} and defined as
\begin{equation}\label{eqn:23}
\mu^{1/\beta},
\end{equation}
where $\mu$ is the order of the method.\\
\end{defn}

Kung and Traub \cite{Kung} presented a hypothesis on the optimality of roots by giving $2^{n-1}$ as the optimal order. This means that the Newton iteration by two evaluations per iterations is optimal with 1.414 as the efficiency index. By taking into account the optimality concept many authors have tried to build iterative methods of optimal higher order of convergence.

This paper is organized as follows: in section 2, we describe the new third-order iterative method by using the concept of inverse function.  In the next section we optimize the method of Chun et. al \cite{Chun} by using the concept of weight function. Finally in the last section we give some physical example and the new methods are compared in the performance with some well known methods.

                             
\section{Development of the method and convergence analysis}
In this section we use the concept of inverse function to derive variants of Bisectrix Newton's Method. In the formula $(\ref{eqn:11a})$, function $y=f(x)$ has been used. Here we use inverse function $x=f^{-1}(y)=g(y)$ instead of $y=f(x)$. Then we have
\begin{eqnarray}\label{eqn:21}
g(y)&=& g(y_n)+\int_{y_n}^{y}g'(s)ds \nonumber \\
    &=&g(y_n)+(y-y_n)\left[\frac{g'(x_n)g'(y_n)+\sqrt{(1+g'(x_n)^2)(1+g'(y_n)^2)}-1}{(g'(x_n)+g'(y_n))}\right],
\end{eqnarray} 
where $y_n=f(x_n)$. Now using the fact that $g'(y)=(f^{-1})^{'}(y)=[f(x)]^{-1}$ and that $y=f(x)=0$, we obtain the following method:
\begin{eqnarray}\label{eqn:22}
 x_{n+1}=x_n-f(x_n)\left[\frac{1+\sqrt{(1+f'(x_n)^2)(1+f'(y_n)^2)}-f'(x_n)f'(y_n)}{f'(x_n)+f'(y_n)}\right].
 \end{eqnarray}
where 
\begin{eqnarray}\label{eqn:22a}
y_n=x_n-\frac{f(x_n)}{f'(x_n)}.
 \end{eqnarray}

Now we prove that order of convergence of this method is also three.

\begin{thm}
Let the function f have sufficient number of continuous derivatives in a neighborhood of $\alpha$ which is a simple root of f, then the method $(\ref{eqn:22})$ has third-order convergence.
\end{thm}
\begin{proof}
Let $e_n=x_n-\alpha$ be the error in the $n^{th}$ iterate and $c_h=\frac{f^{(h)}(\alpha)}{h!}$, $h=1,2,3 . . .$. We provide the Taylor series expansion of each term involved in $(\ref{eqn:22})$. By Taylor expansion around the simple root in the $n^{th}$ iteration, we have\\
\begin{equation}\label{eqn:23}
\begin{split}
f(x_n)=f'(\alpha)[e_n+c_2e_n^2+c_3e_n^3+c_4e_n^4+c_5^5e_n^5+c_6e_n^6+O(e_n^7)]         
\end{split}
\end{equation}
and, we have
\begin{equation}\label{eqn:24}
\begin{split}
f'(x_n)=f'(\alpha)[1+2c_2e_n+3c_3e_n^2+4c_4e_n^3+5c_5^5e_n^4+6c_6e_n^5+O(e_n^6)].           
\end{split}
\end{equation}
Further more it can be easily find 
\begin{equation}\label{eqn:25}
\frac{f(x_n}{f'(x_n)}=e_n-c_2e_n^2+(2c_2^2-2c_3)e_n^3+........+O(e_n^6).
\end{equation}

By considering this relation, we obtain
\begin{equation}\label{eqn:26}
y_n=\alpha+c_2e_n^2+2(c_3-c_2^2)e_n^3+......+O(e_n^{6}).
\end{equation}

At this time, we should expand $f'(y_n)$ around the root by taking into consideration $(\ref{eqn:26})$. Accordingly, we have
\begin{equation}\label{eqn:27}
f'(y_n)=f'(\alpha)[1+2c_2^2e_n^2+(4c_2c_3-4c_2^3)e_n^3+ . . . +O(e_n^{6})].
\end{equation}

By consider the above mentioned relations $(\ref{eqn:23})$, $(\ref{eqn:24})$ and $(\ref{eqn:27})$ in the equation $(\ref{eqn:22})$, we can find
\begin{equation}\label{eqn:28}
\begin{split}
e_{n+1}=\left(\frac{c_2^2}{1+f'(\alpha)^2}+\frac{c_3}{2}\right)e_n^3+O(e_n^{4}).
\end{split}
\end{equation}
It confirms the result.
\end{proof}
\section{Optimal fourth-order method}
By using circle of curvature concept Chun et. al. \cite{Chun} constructed a third-order iterative methods defined by
\begin{eqnarray}\label{eqn:a31}
y_n&=&x_n-\frac{f(x_n)}{f'(x_n)},\nonumber\\
x_{n+}&=&x_n-\frac{1}{2}\left[3-\frac{f'(y_n)}{f'(x_n)}\right]\frac{f(x_n)}{f'(x_n)}.
\end{eqnarray} 
The order of this method three is with three (one derivative and two function) evaluations per full iteration. Clearly its efficiency index $(3^{1/3}\approx 1.442)$ is not high (optimal). We now make use of weight function approach to build our optimal class based on $(\ref{eqn:a31})$ by a simple change in its first step.  Thus we consider
\begin{eqnarray}\label{eqn:a32}
y_n&=&x_n-a\frac{f(x_n)}{f'(x_n)},\nonumber\\
x_{n+}&=&x_n-\frac{1}{2}\left[3-\frac{f'(y_n)}{f'(x_n)}\right]\frac{f(x_n)}{f'(x_n)}\times G(t).
\end{eqnarray}
where  $G(t)$ is a real-valued weight function with  $t=\frac{f'(y_n)}{f'(x_n)}$ and $a$ is a real constant. The weight function should be chosen such that order of convergence arrives at optimal level  four without using more function evaluations. The following theorem indicates under what conditions on the weight functions and constant $a$ in $(\ref{eqn:a32})$, the order of convergence will arrive at the optimal level four:

\begin{thm}
Let the function f have sufficient number of continuous derivatives in a neighborhood of $\alpha$ which is a simple root of f, then the method $(\ref{eqn:a32})$ has fourth-order convergence, when $a=2/3$ and the weight function $G(t)$  satisfies the following conditions
\begin{eqnarray}\label{eqn:a33}
G(1)=1, G^{'}(1)=\frac{-1}{4},  G^{''}(1)=2,  \left|G^{(3)}(1)\right|\leq +\infty,
\end{eqnarray}
and the error equation is given by $(\ref{eqn:36})$.
\end{thm}
\begin{proof}
Using $(\ref{eqn:23})$ and $(\ref{eqn:24})$ and $a=2/3$ in the first step of $(\ref{eqn:a32})$, we have
\begin{eqnarray}\label{eqn:32}
y_n=\alpha+ \frac{e_n}{3}+\frac{2c_2e_n^2}{3}+\frac{4(c_3-c_2^2)e_n^3}{3} + . . . +O(e_n^6).
\end{eqnarray}
Now we should expand $f'(y_n)$ around the root by taking into consideration $(\ref{eqn:32})$.Thus, we have
\begin{equation}\label{eqn:33}
f'(y_n)=f'(\alpha)\left[1+\frac{2c_2e_n}{3}+\frac{(4c_2^2+c_3)e_n^2}{3}+ . . . +O(e_n^{6})\right].
\end{equation}
Furthermore, we have
\begin{eqnarray}\label{eqn:34}
\frac{f'(y_n)}{f'(x_n)}=1-\frac{2c_2}{3}e_n+\left(4c_2^2 - \frac{8c_3}{3}\right)e_n^2+ . . . +O(e_n^{6}).
\end{eqnarray}
By virtue of $(\ref{eqn:34})$ and $(\ref{eqn:a33})$, we attain
\begin{eqnarray}\label{eqn:35}
&&\frac{1}{2}\left[3-\frac{f'(y_n)}{f'(x_n)}\right]\frac{f(x_n)}{f'(x_n)}\times G(t) \nonumber\\
&&=e_n+\left[c_2c_3-\frac{c_4}{9}+-\frac{1}{81}\{309+32H^{(3)}(1)\}c_2^3\right]e_n^4 +O(e_n^{5}).
\end{eqnarray}
Finally using $(\ref{eqn:35})$ in $(\ref{eqn:a32})$, we can have the following general equation, which reveals the fourth-order convergence 
\begin{eqnarray}\label{eqn:36}
&& e_{n+1}=x_{n+1}-\alpha \nonumber\\
&& =x_n-\frac{1}{2}\left[3-\frac{f'(y_n)}{f'(x_n)}\right]\frac{f(x_n)}{f'(x_n)}\times G(t)-\alpha \nonumber\\
&&=\left[-c_2c_3+\frac{c_4}{9}+\frac{1}{81}\{309+32G^{(3)}(1)\}c_2^3\right]e_n^4 +O(e_n^{5}).
\end{eqnarray}
This proves the theorem.
\end{proof}
It is obvious that our novel class of iterations requires three evaluations per iteration, i.e. two first derivative and one function evaluations. Thus our new methods are optimal. Now by choosing appropriate weight functions as presented in $(\ref{eqn:a32})$, we can give optimal two-step methods, such as
\begin{eqnarray}\label{eqn:38}
y_n&=&x_n-\frac{2}{3}\frac{f(x_n)}{f'(x_n)},\nonumber\\
x_{n+}&=&x_n-\frac{1}{2}\left[3-\frac{f'(y_n)}{f'(x_n)}\right]
\left[\frac{9}{4}-\frac{9}{4}\frac{f'(y_n)}{f'(x_n)}+\left(\frac{f'(y_n)}{f'(x_n)}\right)^2\right]\frac{f(x_n)}{f'(x_n)}.
\end{eqnarray}
where its error equation is
\begin{eqnarray}\label{eqn:39}
e_{n+1}=\left[-c_2c_3+\frac{c_4}{9}+\frac{309}{81}c_2^3\right]e_n^4 +O(e_n^{5}).
\end{eqnarray}


\section{Examples.}
In this section we give some physical examples and compare our methods to other well know methods:\\\\
\textit{Example 4.1}\cite{Bradie}
Consider Plank's radiation law
\begin{eqnarray*}
\phi(\lambda)=\frac{8\pi c h \lambda^{-5}}{e^{ch/\lambda kT}-1},
\end{eqnarray*}
where $\lambda$ is the wavelength of the radiation, $t$ is the absolute temperature of the blackbody, $k$ is Boltzmann's constant, $h$ is the Planck's constant and $c$ is the speed of light. This formula calculate the energy density within an isothermal blackbody. Now we want to find wavelength  $\lambda$ which maximize energy density $\phi(\lambda)$. For maximum of $\phi(\lambda)$, it can easily seen that
\begin{eqnarray*}
\frac{(ch/\lambda kT)e^{ch/\lambda kT}}{e^{ch/\lambda kT}-1}=5.
\end{eqnarray*} 
Let $x=ch/\lambda kT$, then it becomes
\begin{eqnarray}\label{eqn:41}
e^{-x}=1-x/5.
\end{eqnarray}
Now the above equation can be rewritten as
\begin{eqnarray}\label{eqn:41}
f_1(x)=e^{-x}-1+x/5.
\end{eqnarray}
Our aim to find the root of the equation $f_1(x)=0$. Clearly zero is its one root, which is not of our interest. If we take $x=5$, then R.H.S. of $(\ref{eqn:41})$ becomes zero and L.H.S. is $e^{-5}\approx 6.74\times 10^{-3}$. This implies one root of the equation $f_1(x)=0$ is near to 5. So that here we compare some well known methods to our methods with initial guess 5.
\begin{table}[htb]
 \caption{Errors Occurring in the estimates of the root of function $f_1$ by the methods described below with initial guess $x_0=5$.}
  \begin{tabular}{llll} \hline
Methods                         &$\left|x_1-\alpha \right|$  &$\left|x_2-\alpha \right|$ & $\left|x_3-\alpha \right|$ \\ \hline 
Newton Method                   & 0.21464e-4 & 0.83264e-11 &0.12530e-23\\ 
Weerakoon   $(\ref{eqn:11d})$   & 0.11208e-6 & 0.37810e-23 &0.14517e-72\\ 
Homeier     $(\ref{eqn:11g})$   & 0.12544e-6 & 0.59456e-23 &0.63310e-72 \\ 
BN Method   $(\ref{eqn:11a})$   & 0.11256e-6 & 0.38466e-23 &0.15352e-72 \\  
Chun Method $(\ref{eqn:a31})$   & 0.98734e-7 & 0.22705e-23 &0.27611e-73 \\ 
Method      $(\ref{eqn:22})$    & 0.11256e-6 & 0.38466e-23 &0.15352e-72\\
Method      $(\ref{eqn:38})$    & 0.42864e-9 & 0.10085e-40 &0.30899e-167 \\                  
\hline
  \end{tabular}
  \label{tab:abbr}
\end{table}

 
\textit{Example 4.2} \cite{Smith} The depth of embedment $x$ of a sheet-pile wall is governed by the equation:
\begin{eqnarray*}
x=\frac{x^3+2.87x^2-10.28}{4.62}.
\end{eqnarray*}
It can be rewritten as
\begin{eqnarray*}
f_2(x)=\frac{x^3+2.87x^2-10.28}{4.62}-x.
\end{eqnarray*}
An engineer has estimated the depth to be $x=2.5$. Here we find the root of the equation $f_2(x)=0$ with initial guess 2.5 and compare some well known methods to our methods .
\begin{table}[htb]
 \caption{Errors Occurring in the estimates of the root of function $f_2$ by the methods described below with initial guess $x_0=2.5$.}
  \begin{tabular}{llll} \hline
Methods                         &$\left|x_1-\alpha \right|$  &$\left|x_2-\alpha \right|$ & $\left|x_3-\alpha \right|$ \\ \hline 
Newton Method                   & 0.85925e-1 & 0.32675e-2 &0.50032e-5\\ 
Weerakoon   $(\ref{eqn:11d})$   & 0.18271e-1 & 0.14770e-5 &0.79610e-18\\ 
Homeier     $(\ref{eqn:11g})$   & 0.49772e-2 &0.33027e-8 &0.95318e-27 \\ 
BN Method   $(\ref{eqn:11a})$   & 0.54594e-2 &0.63617e-8 &0.10016-25 \\ 
Chun Method $(\ref{eqn:a31})$   & 0.27815e-1 & 0.95903e-5 &0.41254e-15 \\ 
Method      $(\ref{eqn:22})$    & 0.54594e-2 & 0.63617e-8 &0.10016e-25\\ 
Method      $(\ref{eqn:38})$    & 0.80338e-2 & 0.15138e-8 &0.19455e-35 \\                  
\hline
  \end{tabular}
  \label{tab:abbr}
\end{table}
\\\\\\\\\\\\\\
\textit{Example 4.3} \cite{Smith} The vertical stress $\sigma_z$ generated at point in an elastic continuum under the edge of a strip footing supporting a uniform pressure $q$ is given by Boussinesq's formula to be:
\begin{eqnarray*}
\sigma_z=\frac{q}{\pi}\{x+Cos x\ Sin x \}
\end{eqnarray*}
A scientist is interested to estimate the value of $x$ at which the vertical stress $\sigma_z$ will be 25 percent of the footing stress $q$. Initially it is estimated that $x=0.4$. The above can be rewritten as for $\sigma_z$ is equal to 25 percent of the footing stress $q$:
\begin{eqnarray*}
f_3(x)=\frac{x+Cos x\ Sin x }{\pi}-\frac{1}{4}.
\end{eqnarray*}
Now we find the root of the equation $f_3(x)=0$ with initial guess 0.4 and compare some well known methods to our methods.
\begin{table}[htb]
 \caption{Errors Occurring in the estimates of the root of function $f_3$ by the methods described below with initial guess $x_0=0.4$.}
  \begin{tabular}{llll} \hline
Methods                         &$\left|x_1-\alpha \right|$  &$\left|x_2-\alpha \right|$ & $\left|x_3-\alpha \right|$ \\ \hline 
Newton Method                   & 0.10737e-3 & 0.50901e-8 &0.11442e-16\\ 
Weerakoon   $(\ref{eqn:11d})$   & 0.20631e-6 & 0.53436e-21 &0.92858e-65\\ 
Homeier     $(\ref{eqn:11g})$   & 0.52795e-6 & 0.19743e-19 &0.10325e-59 \\ 
BN Method   $(\ref{eqn:11a})$   & 0.42239e-7 & 0.13373e-23 &0.42435e-73\\ 
Chun Method $(\ref{eqn:a31})$   & 0.93064e-6 & 0.20624e-18 &0.22446e-56 \\ 
Method      $(\ref{eqn:22})$    & 0.42239e-7 & 0.13373e-23 &0.42435e-73\\ 
Method      $(\ref{eqn:38})$    & 0.25102e-8 & 0.17099e-30 &0.36814e-123 \\                  
\hline
  \end{tabular}
  \label{tab:abbr}
\end{table}

\textsc{Jai Prakash Jaiswal\\
Department of Mathematics, \\
Maulana Azad National Institute of Technology,\\
Bhopal, M.P., India-462051}.\\
E-mail: {asstprofjpmanit@gmail.com; jaiprakashjaiswal@manit.ac.in}.\\\\
\end{document}